\documentclass[11pt]{article}
 
\input{lebStuff_2}

\title{An Elementary Proof of the Dvoretzky--Kiefer--Wolfowitz--Massart Inequality}
\author{Jean-Yves Le Boudec\\EPFL\\jean-yves.leboudec@epfl.ch}

\date{\today}

\begin{document}

\maketitle

\sloppy
\begin{abstract} The Dvoretzky--Kiefer--Wolfowitz--Massart inequality gives an upper bound on the probability that the empirical distribution function of a finite sequence of independent random variables deviates from its theoretical value. It is widely used in statistics in tests and in the production of confidence intervals. The original proof by Massart, later reformulated by Dudley, is very long and technical. Recently, Reeve slightly extends Dvoretzky--Kiefer--Wolfowitz--Massart inequality and presents an alternative, much shorter proof, that uses a continuous time martingale. We present a simpler and less technical proof of the original Dvoretzky--Kiefer--Wolfowitz--Massart inequality, based on a discrete-time martingale and Sion's minimax theorem.
\end{abstract}

\section{Introduction}
Let $X_i$, $i=1:n$ be $n$ independent real-valued random variables with common distribution function $F$ and with empirical distribution function $\hat{F}(x)\eqdef \frac1n\sum_{i=1}^n \ind{X_i\leq x}$. The one-sided Dvoretzky--Kiefer--Wolfowitz--Massart inequality \cite{dvoretzky1956asymptotic,massart1990tight} states the following: 
\begin{theorem}
\be
\P\lp\sup_{x\in\Reals}\lp\hat{F}_n(x)-F(x)\rp> \varepsilon \rp \leq e^{-2n \varepsilon^2}, \forall \varepsilon >0
\label{eq:thm1}
\ee
\label{thm:1}
\end{theorem}
\noindent 
The two-sided Dvoretzky--Kiefer--Wolfowitz--Massart inequality states that
\be  \P\lp\sup_{x\in\Reals}\abs{\hat{F}_n(x)-F(x)}> \varepsilon \rp \leq 2 e^{-2n \varepsilon^2},
\forall \varepsilon >0
\label{eq:thm1bis}
\ee
It can be derived from \eqref{eq:thm1} using the union bound.

In \cite{massart1990tight}, Massart proved \eqref{eq:thm1} for $\varepsilon>\sqrt{\frac{\log 2}{2n}}$ and \eqref{eq:thm1bis} for any $\varepsilon >0$. In \cite{reeve2024short}, Reeve extended the validity of \eqref{eq:thm1} to any $\varepsilon >0$. Note that if $\varepsilon\geq 1$ then the left-handsides in \eqref{eq:thm1} and \eqref{eq:thm1bis} are $0$ and the inequalities trivially hold, hence we assume in the rest that $0<\varepsilon <1$.

The proofs of the original results in \cite{dvoretzky1956asymptotic,massart1990tight} are very long and technical. Dudley \cite[Chapter 1]{dudley2014uniform} provided a more systematic proof, yet it remains very complex and technical. Reeve \cite{reeve2024short} provided a radically different proof of \eqref{eq:thm1}, using a martingale approach. The proof by Reeve removes the constraint $\varepsilon>\sqrt{\frac{\log 2}{2n}}$; it also generalizes the result to a bound on $\P\lp\sup_{x\in \mathcal{I}}\lp\hat{F}_n(x)-F(x)\rp> \varepsilon \rp $ where $\calI$ is any interval of $\Reals$
(instead of the entire $\Reals$), 
but with this generalization the bound is no longer explicit.

While considerably shorter, the proof by Reeve of \eqref{eq:thm1} is still very technical, and can be replaced by a much more elementary one, as we now describe.

%
%
\section{Proof of \thref{thm:1}}
%
%
\subsection{Reduction to the Uniform Case}

\begin{lemma}  
It is sufficient to establish \eqref{eq:thm1} in the case where the common distribution of $X_i$'s is uniform on $[0,1]$.
\end{lemma}
\noindent\textbf{Proof.} See \cite[page 39]{dudley2014uniform}. 
\qed

Let $U_i$, $i=1:n$ be independent random variabiles whose common distribution is uniform on $[0,1]$, and let $\hat{U}_n(t)=\frac1n\sum_{i=1}^n\ind{U_i\leq t}$ be the empirical distribution function. Define
\be
q\eqdef\P\lp\sup_{t\in[0,1]}\lp \hat{U}_n(t)-t\rp > \varepsilon \rp 
\ee
The rest of the proof consists in showing that 
\be
q\leq e^{-2n \varepsilon^2}
\label{theo:dkw}
\ee

\subsection{A Discrete-Time Martingale Approach}
Let $i^*=\floor{n \varepsilon}+1$; note that $i^*\leq n$ because $\varepsilon < 1$. Also let
the function $g(x,\lambda)$ be defined for $x\in \Reals$, $x\geq i^*$ and $\lambda>0$ by
\be
g(x,\lambda)\eqdef n \log(1+\lambda) - x \log\lp1 + \frac{ \lambda}{\frac{x}{n}-\varepsilon}\rp
\label{eq:def_g}
\ee
\begin{lemma}  
\bear
\log q &\leq& 
\inf_{\lambda >0}\max_{x\in [ i^*, n]} g(x,\lambda)
\eqdef \log \tilde{q}
\label{eq:qtilde}
\eear  
\end{lemma}
\noindent\textbf{Proof.}
From the definition of $\hat{U}_n$ it follows that
\be
\lp\sup_{t\in[0,1]}\lp\hat{U}_n(t)-t\rp> \varepsilon \rp \Leftrightarrow 
\lp\exists i \in \lc 1,...,n\rc, U_{(i)}<\frac{i}{n}-\varepsilon\rp
\label{eq:un}
\ee
For $i=0...n$, let $u_i\eqdef \frac{i}{n}-\varepsilon$, so that \eqref{eq:un} implies $q= \P\lp
\bigcup_{i=0}^{n} \lp U_{(i)}< u_i\rp
\rp$. 
For every integer $i=0...n$ define the integer-valued random variable $N_i$ by $N_i=\sum_{k=1}^{n}\ind{U_k<u_i}$ where $\ind{U_k<u_i}=1$ if $U_k<u_i$ and $0$ otherwise. Thus:
$
\lp N_i \geq i\rp 
\Leftrightarrow 
\lp U_{(i)}<u_i\rp
$. 
Observe that $u_i\leq 0$ for $i< i^*$ and $u_i>0$ for $i\geq i^*$, thus:
\be
q= \P\lp
\bigcup_{i=i^*}^{n} \lp U_{(i)}< u_i\rp
\rp
=
 \P\lp
\bigcup_{i=i^*}^{n} \lp N_i \geq i\rp
\rp
\ee
Let $\lambda$ be a positive constant to be optimized upon. Define
\bear
M_{i}&=&\lp 1 +\frac{\lambda}{u_i} \rp^{N_i}  \mfor i={i^*} ... n
\eear
The distribution of $N_i$ is $\mbox{Binomial}(n,u_i)$ therefore
\bear
\esp{M_{i}}&=&\lp1-u_i + u_i \lp 1 +\frac{\lambda}{u_i} \rp\rp ^n
=(1+\lambda)^{n}
\eear
Let $\calF_i$ be the sigma-field generated by $N_{i}, N_{i+1}, ..., N_{n}$. 
The conditional distribution of $N_i$ given $\calF_{i+1}$ depends only on $N_{i+1}$ 
and the conditional distribution of $N_i$ given $\lp N_{i+1}=k\rp$ is $\mbox{Binomial}(k,\frac{u_i}{u_{i+1}})$, thus
\be
\espc{M_i}{\calF_{i+1}}=\lp 1-\frac{u_i}{u_{i+1}} + \frac{u_i}{u_{i+1}}\lp 1 +\frac{\lambda}{u_i}\rp \rp^{N_{i+1}}=M_{i+1}
\ee
hence $(M_i)_{i={i^*}...n}$ is a reverse martingale. It follows, by Doob's martingale inequality:
\bear
\P\lp \bigcup_{i=i^*}^{n} \lp M_i\geq \delta\rp\rp\leq \frac{\esp{M_{i^*}}}{\delta}=\frac{(1+\lambda)^{n}}{\delta}
\eear
Now
$
\lc M_i\geq\delta \rc
=
\lc
N_i \geq \frac{\log\delta}{\log\lp1 + \frac{\lambda}{u_i}\rp}
\rc
$. 
Let 
$\log \delta^*(\lambda)\eqdef \min_{i\in \ib{i^*,n}} \lp i \log\lp1 + \frac{\lambda}{u_i}\rp\rp
$, where $\ib{i^*,n}$ denotes the set of integers $j$ such that $i^*\leq j\leq n$. It comes:
\be
\lc N_i \geq i\rc\subseteq \lc N_i \geq  \frac{\log\delta^*(\lambda)}{\log\lp1 + 
\frac{\lambda}{u_i}\rp}\rc
=\lc M_i \geq \delta^*(\lambda) \rc
, \mfor i\in \ib{i^*,n}
\ee
and thus
\bear
q &\leq & \P\lp \bigcup_{i  = i^*}^{n} \lp M_i\geq \delta^*(\lambda)\rp\rp\leq \frac{\esp{M_{i^*}}}{\delta^*(\lambda)}= \frac{(1+\lambda)^{n}}{\delta^*(\lambda)} 
\eear
and thus
\bear
\log q &\leq& \inf_{\lambda >0}\lp n\log(1+\lambda)-\min_{i\in \ib{{i^*},n}} \lp i \log\lp1 + \frac{\lambda}{u_i}\rp\rp\rp \\
&=&\inf_{\lambda >0}\max_{i\in \ib{{i^*},n}}\lp n\log(1+\lambda)-  i \log\lp1 + \frac{\lambda}{\frac{i}{n}-\varepsilon}\rp\rp 
\eear
Now relax the integrality condition on $i$ (changing $i$ to $x$) and obtain \eqref{eq:qtilde}.
The existence of a $\max$ over $x\in\lb i^*, n \rb$ follows from the continuity with respect to $x$  and the fact that $i^* >n\varepsilon$. \qed

%
%
\subsection{A Saddle-Point Approach}
\begin{lemma}For every $\lambda>0$, the function $x\mapsto g(x,\lambda)$, defined in \eqref{eq:def_g} for $x \in [i^*,n]$, is quasi-concave.
\label{lem:qcave}
\end{lemma}
\noindent\textbf{Proof.} 
Let $e = \varepsilon n$, $\ell = \lambda n$ and $h$ be the function $(0,+\infty)\to \Reals$ such that $h(y)=(y+e)\log y -\log (y+\ell)$. We have $g(x,\lambda)=n \log(1+\lambda)+h(x- \varepsilon n)$. We show next that $h$ is quasi-concave, which will obviously imply that its restriction to $[i^*-\varepsilon n, (1-\varepsilon) n)]$ also is, and thus that $x \mapsto g(x,\lambda)$ is quasi-concave.

By direct computation we find 
\bear
h'(y)
&=&
-  \log\lp y + \ell\rp +  \log y -\frac{y+e}{y + \ell}+\frac{y+e}{y}
\\
h''(y)&=&
-\frac{\ell}{y^2(y+\ell)^2}
\lp
y(2e-\ell) + e\ell
\rp 
\label{eq:der2g}
\eear 
Case 1: $\ell \geq 2e$. Since $y,e,\ell\geq 0$ it follows that $h''(y)\leq 0$ hence $h$ is concave, hence quasi-concave.

\noindent Case 2: $\ell < 2e$. Let $y^*=\frac{e\ell}{\ell-2 e}$, so that $h''(y)<0$ for $0<y<y^*$, $h''(y^*)=0$ and $h''(y)>0 $ for $y>y^*$. Note that 
\bear
\limit{y}{0} h'(y)&=& -  \log \ell -\frac{e}{\ell} +\limit{y}{0} \lp \frac1y\lp y \log y +y+ e\rp 
\rp = +\infty\\
\limit{y}{+\infty} h'(y)&=& \limit{y}{+\infty} \lp \log\frac{y}{ y + \ell}  -\frac{y+e}{y + \ell}+\frac{y+e}{y}\rp=0
\eear
It follows that $h'$ decreases from $+\infty$ to $h(y^*)$ for $y\in (0,y^*]$ and increases from $h(y^*) $ to $0$ for $y \in [y^*,+\infty)$. Thus $h'(y^*)<0$. It follows that there exists some $y^{**}\in (0,y^*)$ such that $h'(y^{**})=0$, $h'(y)>0$ for $y \in (0, y^{**})$ and  $h'(y)<0$ for $y \in (y^{**}, +\infty)$. Thus $h$ is increasing on $(0, y^{**})$ and decreasing on $(y^{**}, +\infty)$, hence is quasi-concave. \qed

\begin{lemma}For every $x \in [i^*,n]$, the function $\lambda\mapsto g(x,\lambda)$, defined in \eqref{eq:def_g} for $\lambda>0$  is quasi-convex.
\label{lem:qcvex}
\end{lemma}
\noindent\textbf{Proof.} 
 \bear\frac{\partial g}{\partial \lambda}(x,\lambda)
 &=&
 \frac{n}{1+\lambda}
 - \frac{xn}{x- \varepsilon n + \lambda n}
 =n\lp\frac{\lambda(n-x)-\varepsilon n}{(1+\lambda)(x- \varepsilon n + \lambda n)}\rp
 \eear

 Recall that $n\varepsilon<i^*\leq x\leq n$. For $x\in \lb i^*, n \rp$, let 
 $\lambda^*(x)\eqdef\frac{\varepsilon n}{n-x}$, so that  
  the map $\lambda \mapsto g(x,\lambda)$ is decreasing for $0< \lambda\leq\lambda^*(x)$ and increasing for $\lambda\geq\lambda^*(x)$, and is thus quasi-convex. 
For $x=n$, $\frac{\partial g}{\partial \lambda}(x,\lambda)<0$ for every $\lambda>0$, the map $\lambda \mapsto g(x,\lambda)$ is decreasing and is thus quasi-convex. 
\qed

\begin{lemma}
The function $g$ defined in \eqref{eq:def_g} satisfies the saddle-point property, namely, $\inf_{\lambda >0}\max_{x\in \lb i^*, n \rb} g(x,\lambda)=\max_{x\in \lb i^*, n \rb}\inf_{\lambda >0} g(x,\lambda)$. 
\label{lemma:saddle}
\end{lemma}
\noindent\textbf{Proof.} Apply Sion's minimax theorem \cite{sion1958general,komiya1988elementary} to $f=-g$. The theorem in \cite{komiya1988elementary} applies to a function $f: X\times \Lambda \to \Reals$ such that (1) $X$ is convex and compact and $\Lambda$ is convex, (2) $x\mapsto f(x,\lambda)$ is quasi-convex and upper-semi-continuous and (3) $\lambda\mapsto f(x,\lambda)$ quasi-concave and lower-semi-continuous. Here, the assumptions hold because of Lemmas~\ref{lem:qcave} and \ref{lem:qcvex}, because $g$ is continuous, $X=[i^*,n]$ is convex and compact and $\Lambda=(0, +\infty)$ is convex. 
\qed

In the following lemma, we use the binary relative entropy function, or binary Kullback-Leibler divergence, defined for $p\in [0,1], q\in(0,1)$ by $\mathrm{kl}(p,q)=p\log\frac{p}{q}+(1-p)\log\frac{1-p}{1-q}$ for $0<p<1$, for $p=0$ by $\mathrm{kl}(0,q)=\log\frac{1}{1-q}$ and for $p=1$ by $\mathrm{kl}(1,q)=\log\frac{1}{q}$.

\begin{lemma}The bound $\tilde{q}$ defined in \eqref{eq:qtilde} satisfies
\be
\log \tilde{q}\leq -n\min_{p\in [0, 1-\varepsilon]}\mathrm{kl}(p+\varepsilon, p) 
\ee
\label{lemma:kl}
\end{lemma}
\noindent\textbf{Proof.} Apply \lref{lemma:saddle} 
and compute $g^*(x)\eqdef\inf_{\lambda >0} g(x,\lambda)$. 
For $x\in \lb i^*, n \rp$, we have obtained that $g(x,.)$ has a minimum at $\lambda=\lambda^*(x)$ hence 
\bear
g^*(x)&=&g(x,\lambda^*(x))
= n \log\lp 1+ \frac{\varepsilon n}{n-x}\rp - x \log\lp
1+ \frac{\frac{\varepsilon n}{n-x}}{\frac{x}{n}-\varepsilon}\rp 
\\
\frac{g^*(x)}{n}
&=&
\log\lp 1-\frac{x}{n}+\varepsilon \rp -\log \lp 1-\frac{x}{n}\rp \nonumber\ \\
&&+ \frac{x}{n} \log\lp \frac{x}{n}-\varepsilon \rp +\frac{x}{n} \log \lp 1 - \frac{x}{n} \rp - \frac{x}{n} \log \frac{x}{n} -\frac{x}{n} \log \lp 1-\frac{x}{n} + \varepsilon\rp\\
&=&-\frac{x}{n}\log\frac{\frac{x}{n}}{\frac{x}{n}-\varepsilon} - \lp 1- \frac{x}{n} \rp \log\frac{1-\frac{x}{n}}{1-\frac{x}{n}+\varepsilon}=-\mathrm{kl}\lp\frac{x}{n},\frac{x}{n}-\varepsilon\rp
\eear

For $x=n$, $g(x,.)$ is decreasing and 
\bear
g^*(n)&=&\limit{\lambda}{\infty}g(n,\lambda)=\limit{\lambda}{\infty}n\log\frac{1+\lambda}{1+\frac{\lambda}{1-\varepsilon}}=n\log (1-\varepsilon)=-n \;\mathrm{kl}(1, 1-\varepsilon)
\eear
In summary, we have
\be
g^*(x)=-n\;\mathrm{kl}\lp\frac{x}{n},\frac{x}{n}-\varepsilon\rp \mfor x\in \lb i^*, n \rb
\ee
Hence
\be
\log \tilde{q}= - n \min_{x \in [i^*,n]}\mathrm{kl}\lp\frac{x}{n},\frac{x}{n}-\varepsilon\rp \leq -n\min_{p\in [0, 1-\varepsilon]}\mathrm{kl}(p+\varepsilon, p) 
\ee
\qed
%
%
\subsection{Proof of \thref{thm:1}}

By Pinsker's inequality \cite{yeung2008information}, $\mathrm{kl}(p+\varepsilon, p) \geq 2 \varepsilon^2$ hence
\be
\log \tilde{q}\leq -2n\varepsilon^2 
\ee
\qed


\bibliographystyle{unsrt}

\end{document}